\documentclass[reqno]{amsart}
\usepackage[utf8]{inputenc}

\usepackage{amsfonts,color,newclude}
\usepackage{amsthm}
\usepackage{amssymb,transparent}
\usepackage{amsmath}
\usepackage{amscd}
\newcommand{\trans}{^\top}
\usepackage{todonotes}

\usepackage{graphicx}
\usepackage{subfiles}
\usepackage{tikz}
\usepackage{pgffor} 
\usetikzlibrary{arrows, positioning, calc}
\usepackage{tikz-cd}
\usepackage{MnSymbol}

\usepackage[colorlinks=true, pdfstartview=FitV, linkcolor=blue, citecolor=blue, urlcolor=blue]{hyperref}
\usepackage{hyperref}
\usepackage[capitalize,nameinlink,noabbrev]{cleveref}

\makeatletter
\renewcommand*\env@matrix[1][c]{\hskip -\arraycolsep
  \let\@ifnextchar\new@ifnextchar
  \array{*\c@MaxMatrixCols #1}}
\makeatother

\newtheorem{theorem}{Theorem}[section]

\newtheorem{proposition}[theorem]{Proposition}
\newtheorem{lemma}[theorem]{Lemma}
\newtheorem{cor}[theorem]{Corollary}

\newtheorem{question}[theorem]{Question}
\newtheorem*{question*}{Motivating Questions}
\newtheorem{conj}[theorem]{Conjecture}
\theoremstyle{definition}

\theoremstyle{remark}

\author[Klanderman, Montee, Piotrowski, Rice, Shader]
{Sarah Klanderman$^1$, MurphyKate Montee$^2$, Andrzej Piotrowski$^3$, Alex Rice$^4$, and Bryan Shader$^5$}
\thanks{$^1$ Marian University, sklanderman@marian.edu\\
\phantom{aa} $^2$ Carleton College, mmontee@carleton.edu
\\
\phantom{aa,}$^3$ University of Alaska Southeast, apiotrowski@alaska.edu\\
\phantom{aa,}$^4$ Millsaps College,  riceaj@millsaps.edu
\\
\phantom{aa,}$^5$ University of Wyoming, bshader@uwyo.edu
}
\date{}
\title{Determinants of Seidel Tournament Matrices}
\begin{document}

\begin{abstract}
The Seidel matrix of a tournament on $n$ players is an $n\times n$ skew-symmetric matrix with entries in $\{0, 1, -1\}$ that encapsulates the outcomes of the games in the given tournament. It is known that 
the determinant of an $n\times n$ Seidel matrix is $0$ if $n$ is odd, and is an odd perfect square if $n$ is even.  This leads to 
the study of the set
\[ \mathcal{D}(n)= \{ \sqrt{\det S}: \mbox{ $S$ is an $n\times n$ Seidel matrix}\}.  \]
This paper studies various questions about $\mathcal{D}(n)$.
It is shown that  $\mathcal{D}(n)$ is a proper subset of $\mathcal{D}(n+2)$ for every positive even integer, and  
every odd integer in the interval $[1, 1+n^2/2]$ is in $\mathcal{D}(n)$ for $n$ even.
The expected value and variance of $\det S$ over the $n\times n$ Seidel matrices chosen uniformly at random is determined, and 
upper bounds on $\max \mathcal{D}(n)$ are given, and related to the Hadamard conjecture.
Finally, it is shown that for infinitely many $n$, $\mathcal{D}(n)$
contains a gap (that is, there are odd integers $k<\ell <m$ 
such that $k, m \in \mathcal{D}(n)$ but $\ell \notin \mathcal{D}(n)$) and several properties of the characteristic polynomials of Seidel matrices are established. 
\end{abstract}
\maketitle


\section{Introduction}
In combinatorics, a \textit{tournament of order $n$} is a digraph $T$ that represents the results of a competition involving $n$ players,
with one game between each pair of players and no ties. Typically, we take the $n$ players to be $1$, \ldots, $n$. 
In this setting $T$ has vertices $1$, \ldots, $n$ and an arc $ij$ from $i$ to $j$ provided player $i$ 
beats player $j$ in the game between $i$ and $j$, and for each $i\neq j$ exactly one of the arcs $ij$ or $ji$ is present in the tournament $T$.   

\begin{figure}[htbp]
    \centering
    \includegraphics[width = .2\textwidth]{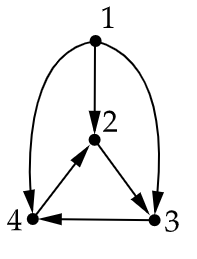}
    \caption{The diamond tournament.}
    \label{fig:diamond}
\end{figure}

Associated with $T$ are two $n\times n$ matrices.  The \textit{adjacency matrix} of $T$ is the $n\times n$ matrix
$A=[a_{ij}]$ where $a_{ij}=1$ if $ij$ is a directed arc in $T$ and $a_{ij}=0$ otherwise.  The \textit{Seidel matrix} of $T$
is the $n\times n$ skew-symmetric matrix $S=[s_{ij}]$ where $s_{ii}=0 $ for $i=1, \ldots, n$, $s_{ij}=1$ if $ij$ in $T$, 
and $s_{ij}=-1$, otherwise. 
For example, the tournament in \Cref{fig:diamond} has adjacency matrix $A$ and Seidel matrix $S$, where

    \[
    A = \begin{bmatrix}[r]
        0 & 1 & 1 &1 \\
        0 & 0 & 1 & 0 \\
        0 & 0 & 0 & 1 \\
        0 & 1 & 0 & 0
        \end{bmatrix}
\qquad\mbox{ and } \qquad 
    S = \begin{bmatrix}[r]
        0 & 1 & 1 & 1 \\
        -1& 0 & 1 & -1\\
        -1& -1& 0 & 1 \\
        -1& 1 & -1 & 0
        \end{bmatrix}.
    \]
The adjacency matrix and Seidel matrix of $T$ are related by the following:
\begin{itemize}
\item[(a)] $A+A{\trans}=J-I$, where  $J$ is the all ones matrix and $I$ is the identity matrix of appropriate size, and
\item[(b)] $S=A-A{\trans}=2A+I-J$.
\end{itemize}
Historically, the adjacency matrix has been used to study tournaments (see for example the results in the delightful book 
\cite{Moon}, and papers like \cite{deCaen,KS} that study the spectra of these adjacency matrices). More recently, the Seidel matrix has become a fruitful area of research. For example, \cite{BSTZ} studies \emph{$k$-spectrally monomorphic} adjacency and Seidel matrices (i.e., matrices for which the $k\times k$ principal submatrices all have the same characteristic polynomial), and \cite{BC} studies Seidel matrices for which the principal minors are bounded by a fixed integer $k$, for various values of $k$. The \emph{energy} of a Seidel matrix, that is the sum of the moduli of its eigenvalues, is studied in \cite{Ito}.  Seidel matrices are also related to \emph{skew-conference matrices}, that is, $n\times n$ skew-symmetric matrices $M$ with entries in $\{0, \pm 1\}$ and $MM\trans = (n-1)I$. 

In this paper we explore the set of all possible determinants of Seidel matrices of order $n$. Using SageMath, we compute these sets for small $n$. The data is shown in \cref{fig: determinant sets}. Some properties of these sets follow immediately from the structure of a Seidel matrix. For example, since each Seidel matrix $S$ of order $n$ is skew-symmetric, we know that $\det S = \det S{\trans} = (-1)^n\det S$, so $\det S = 0$ whenever $n$ is odd. Since eigenvalues of real skew-symmetric matrices occur in conjugate pairs and are purely imaginary, $\det S \geq 0$. Furthermore, if $J$ denotes the all ones matrix, we see that $S \equiv J - I \mod 2$. But $\det(J-I) = (-1)^{n-1}(n-1)$, so $\det S \equiv n-1\mod 2$. Hence, for all even $n$, $\det S$ is odd. Furthermore, a classical result of Cayley \cite{Cayley} shows that if the matrix $K$ is skew-symmetric, then the determinant of $K$ is the square of the \emph{Pfaffian} of $K$. Thus $\det(S)$ is the square of an odd integer.

\begin{figure}
    \centering
    \begin{tabular}{l|l}
     $n$&  possible determinants of Seidel matrices of order $n$\\
     \hline
     1& 0\\
     2& 1\\
     3& 0\\
     4& $1, 3^2$\\
     5& 0\\
     6& $1, 3^2, 5^2, 7^2, 9^2$\\
     7& 0\\
     8& $1, 3^2, 5^2, \dots, 27^2; 31^2, 33^2, 35^2; 49^2$\\ 
     9& 0\\
     10& $1, 3^2, 5^2, \dots, 129^2; 133^2, 135^2, 137^2, \dots, 143^2; 147^2; 153^2; 161^2; 165^2; 175^2; 183^2$ \\
     11& 0\\
     12& $1^2, \dots, 701^2; 705^2, \dots, 723^2; 727^2,  \dots, 751^2; 755^2, \dots, 769^2; 773^2, \dots, 781^2;$ \\
        &$789^2, \dots, 795^2; 799^2; 803^2, 805^2, 807^2; 811^2; 817^2; 825^2, \dots 835^2; 839^2, 841^2;$ \\
        &$ 847^2; 855^2; 861^2; 867^2, 869^2; 873^2; 877^2; 891^2; 905^2; 931^2; 945^2; 979^2; 1089^2; 1331^2$
\end{tabular}
    \caption{ }
    \label{fig: determinant sets}
\end{figure}

The following question is asked in \cite{BBCL}:
\begin{question}\label{question}
    For all odd positive $k$, does there exist some $n_k$ so that $k \in \mathcal{D}(n_k)$?
\end{question}
We answer this in the affirmative in \Cref{sec: small values}.

More generally, we define and study
    \[
        \mathcal{D}(n) = \{\sqrt{\det S} \mid S \mbox{ is a Seidel matrix of order }n\}.
    \]
The data in \cref{fig: determinant sets} leads to the following additional questions.

\begin{question*} \rm Let $n$ be a positive even integer.
    \begin{itemize}
    \item[(a)] Is it the case that $\mathcal{D}(n) \subseteq \mathcal{D}(n+2)$? If so, is the containment strict?
    \item[(b)] Is $1$ the smallest element of $\mathcal{D}(n)$? What is the largest element of $\mathcal{D}(n)$? 
    \item[(c)]  How are values in $D(n)$ distributed?
    \item[(d)] For $n\geq 8$ is there always a gap in $\mathcal{D}(n)$, i.e., an odd $k$ such that $k\notin \mathcal{D}(n)$
and $k$ is between the minimum and maximum element of $\mathcal{D}(n)$? If there is a gap, what is the smallest (or largest) such $k$?
    \item[(e)] Can one give  necessary and sufficient conditions for an integer to be in $\mathcal{D}(n)$?
\end{itemize}
\end{question*}

The remainder of this paper is organized around these questions.  In \cref{sec: montonicity} we give ways to compute determinants of Seidel matrices from determinants of related matrices, and we use these techniques to answer both parts of Question (a) in the affirmative. In \cref{sec: small values} we show that when $n$ is even, 
each odd integer in the interval $[1,1+n^2/2]$
is in $\mathcal{D}(n)$, addressing the latter half of Question (d). In \cref{sec: average}, we address Question (c) by calculating the expected value of the determinant of a Seidel matrix chosen uniformly at random, and give a recursive formula for the variance. 
In \cref{sec: large}, we
give a partial answer to the latter half of Question (b), and 
we relate Question (e) to the Hadamard conjecture. 
In \cref{sec: gaps} we study Question (d) and show that for infinitely many values $n$ there is a gap in $\mathcal{D}(n)$. In \Cref{sec: charpoly} we investigate possible characteristic polynomials of Seidel matrices, with a focus on principal submatrices of skew-conferences matrices, and we calculate the expected characteristic polynomial of a Seidel matrix chosen uniformly at random.

\subsection*{Acknowledgements}
The authors would like to thank the American Institute of Mathematics (AIM) and the REUF program (NSF DMS-2015462) for supporting this work.  The authors would also like to thank Xin Tang of Fayetteville State University for his participation in the project at the AIM REUF workshop.

\section{Monotonicity of $\mathcal{D}(n)$}\label{sec: montonicity}

The primary goal of this section is to show that $\mathcal{D}(n)\subsetneq \mathcal{D}(n+2)$ for all even $n$. We first show that $\mathcal{D}(n)\subseteq\mathcal{D}(n+2)$. 
To do this we will use the Sherman-Morrison-Woodbury formula 
(see \cite[Theorem 10.11]{BRoy}), and the following lemma. 

\bigskip\noindent
{\bf Sherman-Morrison-Woodbury formula.}
\\
{\it Let $A$ be an invertible $n\times n$ matrix, and $X$ 
and $Y$ be $n\times k$ matrices. Then 
\[ \det (A +XY\trans)= \det A \cdot \det ( I+Y\trans A^{-1}X).
\]
}

\bigskip\noindent
If $S$ is an invertible $n\times n$ matrix and  
$W$ is an $n\times n$ rank-one matrix, then $W=xy\trans$ for some 
vectors $x$ and $y$, and hence by the Sherman-Morrison-Woodbury 
formula $\det (S + tW)= \det S (1+ ty\trans S^{-1} x) $ is a linear function in $t$. We use this observation in the next lemma. 

\begin{lemma}
\label{lem:skew-non}
Let $T$ be a tournament of order $n$ with $n$ even, and let $A$ and $S$ be the adjacency and Seidel matrices of $T$, respectively. Then $\det S= \det (2A+I)$.
\end{lemma}

\begin{proof}
Note that 
    \begin{equation}\label{det-form1}
        S+J=2A+I.
    \end{equation}
Transposing and negating both sides of (\ref{det-form1}) gives 
    \begin{equation}\label{det-form2}
        S-J= -(2A\trans+I).
    \end{equation}
Taking determinants of both sides of (\ref{det-form1}) and (\ref{det-form2}) and using the fact that $n$ is even, we get 
    \begin{equation}\label{det-form3}
        \det (S+J)=\det (S-J).
    \end{equation}
Let $f(t)= \det (S+tJ)$ for $t\in \mathbb{R}$. As $J$ is a rank 1 matrix, $f(t)$ is a linear function of $t$. The linearity of $f$ and (\ref{det-form3}) imply $f(t)= \det S$ for all $t$. In particular, $\det (2A+I)= f(1)=\det S$. \qedhere
\end{proof}

Note in particular, if $R_k$ is the Seidel matrix corresponding to the \textit{transitive tournament} with $k$ players (that is, $R_k$ is the skew-symmetric matrix with all ones above the main diagonal), then \cref{lem:skew-non} implies that $\det R_k=1$ when $k$ is even.

We also need a construction introduced in 
\cite{BBCL} that builds a new tournament from two tournaments of smaller order.
Let $T_1$ and  $T_2$ be tournaments. The \emph{join of $T_1$ on $T_2$}, denoted $T_1 \to T_2$, is the tournament obtained from the disjoint union of $T_1$ and $T_2$ by adding each arc $ij$ from a vertex in $T_1$ to a vertex in $T_2$. Note that if $S_1$ and  $S_2$ are the Seidel matrices of $T_1$ and  $T_2,$ respectively, then the Seidel matrix of $T_1\to T_2$ is the block matrix

\[ S = \left[ \begin{array}{r|r}
        S_1 &J \\ \hline
        -J\trans & S_2
    \end{array} \right]
    \]

The following lemma is proven in \cite{BBCL}. 
Here we give a different proof; one that uses 
\cref{lem:skew-non}.

\begin{lemma}\label{lem:seidel of a join}
    Let $T_1$ and  $T_2$ be tournaments, with Seidel matrices $S_1$ and  $S_2$. Let $S$ be the Seidel matrix of the join $T_1 \to T_2$. 
    If at least one of $T_1$ and $T_2$ has even order, then 
        \[\det\; S = \det\; S_1 \cdot \det\; S_2.\]
\end{lemma}

\begin{proof}
    If the orders of $T_1$ and $T_2$ are of different parity, then 
    one of $S_1$ and $S_2$ has odd order, and $S$ has odd order.
    Hence, both $\det\, S_1\cdot \det\, S_2$ and $\det\, S$ equal $0$.

    Now assume that both $T_1$ and $T_2$ have even order. 
    Let $A_i$  denote the adjacency matrices of $T_i$ for $i \in \{1, 2\}$. Let $A$ and $S$ denote the adjacency and Seidel matrices, respectively, of $T_1 \to T_2$. Then 
        $$  A= \left[ \begin{array}{c|c}
            A_1 & J \\ \hline
           O & A_2 \end{array} 
            \right],
        $$
    where $J$ is the matrix of all ones and $O$ is the matrix of all zeroes. Thus $2A+I$ is block upper triangular, and we have
        \[
            \det\, S = \det (2A+I)= \det (2A_1+I) \det (2A_2+I)= \det\,S_1 \cdot \det\, S_2,
        \]
by \cref{lem:skew-non}.
\end{proof}

\begin{cor}\label{cor:contain}
    Let $n=k +\ell$ where $k$ and $\ell$ are even positive integers. Then $\mathcal{D}(k)\mathcal{D}(\ell) \subseteq \mathcal{D}(n)$. In particular, for each positive even integer $n$, $\mathcal{D}(n) \subseteq \mathcal{D}(n+2)$.
\end{cor} 

\begin{proof}
Let $S_k$ be the Seidel matrix of a tournament of order $k$, and let $S_{\ell}$ be the Seidel matrix of a tournament of order $\ell$.  Then $S_k \rightarrow S_{\ell}$
is a Seidel matrix of a tournament of order $k + \ell$ having determinant $\det S_k \det  R_{\ell}$ by \Cref{lem:seidel of a join}. Thus, $\mathcal{D}(k) \mathcal{D}(\ell) 
\subseteq \mathcal{D}(n)$.
Taking $k=n$ and $\ell=2$ gives
$D(n) \subseteq D(n+2)$ for each positive even integer $n$. 
\end{proof}

We show that $\mathcal{D}(n) \neq \mathcal{D}(n+2)$ with a more subtle construction.  The tournaments $T$ and  $T'$ are \emph{switching equivalent} if there exists a subset of teams $\alpha$ so that changing the orientation on every edge $ik$, where $i \in \alpha$ and $k \notin \alpha$, turns $T$ into $T'$. Equivalently, if $S$ and $S'$ are the Seidel matrices of $T$ and $T'$, then $T$ and $T'$ are switching equivalent if and only if there is a diagonal matrix $D$, each of whose diagonal entries is $\pm 1$, such that $DSD=S'$.
Thus, switching equivalent tournaments  have Seidel matrices with the same determinants (though possibly opposite Pfaffians). 
It follows that for $n\geq 4$, there are tournaments of order $n$ that are not switching equivalent.  If instead we are able to change a single arc at a time, we can turn any tournament into any other of the same order by reversing a set of arcs. 

Let $T$ be a tournament of order $n$. The \emph{$(i,j)$-reversal of $T$} is the tournament $T'$ whose edges have  the same orientations as $T$ except the edge joining $i$ and $j$  which has the reverse orientation. The following lemma gives a formula to find the determinant of the $(i,j)$-reversal of a tournament.

\begin{lemma}\label{lem: inverse-det property}
Let $S=[s_{k\ell}]$ be the Seidel matrix for a tournament $T$. Let $T'$ be the $(i,j)$-reversal of $T$. Let $S'$ be the Seidel matrix for $T'$, and suppose that $s_{ij} = 1$. Then
    \[
        \det S'  = \det S \cdot (1+ 2S^{-1}_{ij})^2.
    \]
In particular,
    \begin{enumerate}
        \item[\rm (a)]$\det S' >\det(S)$ if $S^{-1}_{ij} >0$ or $S^{-1}_{ij}<-1$, 
        \item[\rm (b)] $\det S' = \det S $ if $S^{-1}_{ij} \in \{0, -1\}$, and 
        \item[\rm (c)] $\det S' <\det S $ if $-1<S^{-1}_{ij}<0$,
    \end{enumerate}
    where $S^{-1}_{ij}$ denotes the $(i,j)$-entry of $S^{-1}$.
\end{lemma}

\begin{proof}
Note that $S'$ is a sum of $S$ and a matrix 
$P=[p_{k\ell}]$, whose entries are $0$ except $p_{ij} = -2$ and  $p_{ji} = 2$. In particular, let $e_i$ and  $e_j$ be the $i$-th and $j$-th standard basis vector of appropriate order. Then 
    \[
        S' = S + 2\left[ \begin{array}{rr}
                                                e_i&e_j\\
                                               \end{array} \right] 
                                               \left[ \begin{array}{r}
                                                   -e_j\trans 
                                                   \\[2pt]
                                                   e_i\trans
                                               \end{array} \right].
    \]
Applying the Sherman-Morrison-Woodbury  formula, we get
        \begin{align*}
            \det S' 
                &= \det S \cdot \det \left( I + 2\begin{bmatrix}
                                                   -e_j\trans
                                                   \\
                                                   e_i\trans
                                               \end{bmatrix} 
                                               S^{-1}
                                               \begin{bmatrix}
                                                e_i&e_j\\
                                               \end{bmatrix} \right)\\
                &=\det S \cdot 
                \det\left(I + 2\begin{bmatrix} -e_j\trans S^{-1} e_i &-e_j\trans S^{-1}e_j \\
                                                    e_i\trans S^{-1} e_i &e_i\trans S^{-1}e_j
                                                    \end{bmatrix}
                                                    \right) \\
                &=\det S \cdot\det\left(I + 2\begin{bmatrix} -S^{-1}_{ji} &-S^{-1}_{jj} \\
                                                    S^{-1}_{ii} &S^{-1}_{ij}
                                                    \end{bmatrix}
                                                    \right). 
        \end{align*}
Since $S$ and $S^{-1}$ are skew-symmetric, $S_{ij}^{-1} = -S_{ji}^{-1}$ and $S^{-1}_{ii} = S^{-1}_{jj} = 0$. Therefore
    \begin{align*}
        \det\, S' &= \det\, S \cdot \det \begin{bmatrix} 1+2S^{-1}_{ij} &0 \\
                                                    0 &1+2S^{-1}_{ij}
                                                    \end{bmatrix}
                                                     \\
                 &= \det\; S \cdot (1 + 2S^{-1}_{ij})^2.           \qedhere                        
    \end{align*} 
\end{proof}

\begin{cor}
Let $S$ be the $n\times n$ Seidel matrix of a tournament $T$, where $n$ is even.  Then there exists an $(n+2)\times(n+2)$ Seidel matrix with strictly larger determinant than $\det S$.
\end{cor}

\begin{proof}

Let $T_2$ be the tournament with Seidel matrix $S_2 = \left[ \begin{array}{rr} 0&1\\ -1&0 \end{array} \right]$, and let $B$ be the Seidel matrix of $T_2 \to T$. Then 
    \[
    B = \left[ \begin{array}{c|c} {S_2} & J
\\ \hline
-J\trans & S \end{array} \right]
    \]
and by \Cref{lem:seidel of a join}, $\det B = \det S$. 
Direct calculation shows that 
\[
B^{-1} = \left[ \begin{array}{c|c} {S_2}^{-1} & -{S_2}^{-1} JS^{-1}
\\ \hline
S^{-1} J\trans {S_2}^{-1} & S^{-1} \end{array} \right].
\]

Since $J(-S_2^{-1} JS^{-1}) = 0$, each column-sum of $-S_2^{-1}J S^{-1}$ is zero. On the other hand, since $J$ is nonzero we have $-S_2^{-1}JS^{-1}$
is nonzero. Thus, some entry of $-S_2^{-1} JS^{-1}$ is positive.
Every corresponding entry of $B$ is 1, so by \Cref{lem: inverse-det property} reversing the arc corresponding to that entry results in an $(n+2)\times (n+2)$ Seidel matrix whose determinant is greater than that of $B$, and hence that of $S$.
\end{proof}

\section{Small Values of $\mathcal{D}(n)$} \label{sec: small values}

In this section we investigate which small values of $k$ belong to $\mathcal{D}(n)$, $n$ even.

Recall that $R_n$ denotes the Seidel matrix of the transitive tournament  of order $n$. The basic {\it nega-circulant matrix} $W_n$ of order $n$ is the $n\times n$ matrix with one in positions $(1,2)$, \ldots, $(n-1,n)$; negative one in position $(n,1)$; and zeroes elsewhere.  
For each integer $m$ with $1\leq m \leq n-1$, 
 it can be seen that ${W_n}^m$ is the matrix with one in positions 
$(1,m+1)$, 
\ldots, $(n-m,n)$; negative one in positions $(n-m+1, 1)$, 
\ldots, $(n, m)$ and zeroes elsewhere.  Additionally, ${W_n}^n=-I$.

\begin{lemma}\label{lem: det rn is 1}
The Seidel matrix $R_n$ of the transitive tournament of order $n$ satisfies
\[\det R_n=\left\{ \begin{array}{rl} 0 & \mbox{ if $n$ is odd, and}\\
1 & \mbox{ if $n$ is even.} \end{array} \right. 
\]
Moreover, if $n$ is even then 
\[ 
      R_n^{-1} =  \left[ \begin{array}{rrrrrrr}
            0 & -1 & 1 & \cdots & 1 & -1 \\
            1 & 0 & -1 & 1 &  & 1\\
            -1 & 1 & \ddots & \ddots & \ddots & \vdots\\
            \vdots &\ddots & \ddots & \ddots & \ddots & 1 \\
            -1 & & \ddots& 1 & 0 & -1\\
            1 & -1 & \cdots & -1 & 1 & 0
        \end{array}\right] .
    \] 
\end{lemma}
\begin{proof}
As $R_n$ is skew-symmetric, if $n$ is odd then $\det R_n=0$.

Now assume $n$ is even. 
By \Cref{lem:skew-non}, $\det R_n=\det (I+2A_n)=1$
where $A_n$ is the $n\times n$ matrix with ones above the diagonal and zeros elsewhere.
Thus $R_n$ is invertible, which by the Cayley-Hamilton theorem 
implies that ${R_n}^{-1}$ is a polynomial in $R_n$. Furthermore, as $R_n$ is a polynomial in $W_n$, ${R_n}^{-1}$ is a polynomial in $W_n$.
Because  
    \[ 
        \left[ \begin{array}{rrrrrrr} 
            0& -1 & 1 & -1 &\cdots & 1 & -1 
                    \end{array} \right] R_n 
        = \left[ \begin{array}{ccccc}
                1 & 0 & \cdots & 0 \end{array} \right],
    \]
the first row of $R_n^{-1}$ is 
$\left[ \begin{array}{ccccccc} 
            0& -1 & 1 & -1 &\cdots & 1 & -1 
                    \end{array} \right]$.

As $R_n^{-1}$ is a polynomial in $W_n$,  $R_n^{-1}= -{W_n}^1+{W_n}^2 -{W_n}^3 +\cdots +(-1)^{n-1}W_{n}^{n-1}$. The result now follows from formula for the powers of $W_n$.
\end{proof}

By \Cref{lem: det rn is 1}, we know that for $n$ even, the smallest value of $\mathcal{D}(n)$ is $1$. The next result shows that all odd integers between 1 and $n^2/2 +1$ are in $\mathcal{D}(n)$ when $n$ is even. In particular, this answers \Cref{question} in the affirmative.

We use the following result about the Schur complement in the argument (see \cite[Statement 0.8.5]{HJ}).

\bigskip
\noindent
{\bf Schur complement}.
\\
{\it 
    Assume that  $A$ and $D$ are square matrices with $D$ invertible, and let $M$ be the block matrix
            \[
                M = \left[ \begin{array}{c|c} 
                                A & B\\ \hline
                                C & D 
                        \end{array} \right].
            \]
    The \emph{Schur complement of $D$} is the matrix $M/D:= A - BD^{-1}C$. 
    
    Furthermore
            \[
                \det(M) = \det(M/D)\det(D).
            \]
}
\bigskip
\noindent

\begin{theorem} \label{thm: quadratic}
For each even positive integer $n$ and each positive odd integer $k < n^2/2+1$ there is a Seidel tournament matrix of order $n+2$ whose determinant is $k^2$.
\end{theorem}

\begin{proof}

Let $x{\trans} = 
                \left[ \begin{array}{cccccc}
                    1 & -1 & 1 & \cdots & 1 & -1 \end{array} \right] $
be the $1\times n$ vector of alternating ones and negative ones, 
and let $y \in \{ \pm 1\}^n$.
The matrix
    \[
        S= \left[ \begin{array}{cc|c}
                 0 & 1 & x{\trans}\\ 
                -1 & 0 & y{\trans}\\ \hline
                 -x & -y & R_n \end{array} \right]. 
    \] 
is a Seidel matrix of a tournament of order $n+2$.

By the Schur complement of $S$ and \cref{lem: det rn is 1}, we have
\begin{eqnarray*}
\det S &= &\det(S/R_n)\det(R_n)\\
&= & \det \left( \left[ \begin{array}{cc}
0 &1 \\ -1 & 0 \end{array}
 \right] + \left[ \begin{array}{c} x\trans \\ y\trans 
 \end{array} \right] R_n^{-1} \left[ \begin{array}{cc} 
 x & y \end{array} \right] \right)  \\
 & = & 
 \det 
 \left(
 \left[
 \begin{array}{cc} 0 & 1 \\ -1 & 0 \end{array} 
 \right] + 
 \left[ \begin{array}{cc} 0 & x\trans R_n^{-1} y \\
 -x\trans R_n^{-1}y & 0 \end{array} \right] 
 \right)\\
 &= & (1+x\trans R_n^{-1}y)^2,
 \end{eqnarray*}
 with the penultimate equality coming from the fact that 
 $R_n^{-1}$ is skew-symmetric and therefore $y\trans R_n^{-1}x = (y\trans R_n^{-1}x)\trans = -x\trans R_n^{-1}y$. 
 
 By \Cref{lem: det rn is 1} we have
 \[ 
 x\trans R_n^{-1}=
 \left[ \begin{array}{rrrrrrrrrr}
 -(n-1) & n-3 & -(n-5) & \cdots & \pm 1 & \pm 1 & \cdots 
 & -(n-5) & n-3 & -(n-1) \end{array} \right] .
 \]
Now let 
     \[
            \mathcal{L}_n= \{ x\trans R_n^{-1}y \mid 
                    y \in \{\pm 1\}^n \mbox{ and } x\trans R_n^{-1} y\geq 0  \}.
    \] 
 We show that   
 $\mathcal{L}_n=\{0,2, \ldots, n^2/2\}$ for all even $n$ by induction.
For $n=2$, we have 
    \[
        x\trans R_2^{-1} = \left[ \begin{array}{cc} -1 & -1 \end{array} \right] \left[ \begin{array}{c} y_1 \\ y_2 
            \end{array} \right]
    \]
where $y_i \in \{\pm 1\}$. 
The only nonnegative values that can be achieved are $0$ and $2$, 
as desired.

Assume that the result holds for $n$, and consider $\mathcal{L}_{n+2}$.
Every element of $\mathcal{L}_{n+2}$ is twice the sum of the elements in a subset of $\{1,3,\ldots, n+1\}$, and hence is contained in $\{0,2, \ldots, 2((n+2)/2)^2=(n+2)^2/2\}$. For each $k \in \mathcal{L}_n$, let $y_k$ denote the vector satisfying $x\trans R_n^{-1}y_k = k$. Let $Y_k$ denote the vector obtained by bordering $y_k$ with $-1$ on its top and $1$ on the bottom, i.e., 
$Y_k = \left[ \begin{array}{r} -1 \\ y_k \\  1 \end{array} \right]$.
Then 
    \[
        x\trans R_{n+2}^{-1} Y_k = (n-1) + x\trans R_{n+2}^{-1} y_k + -(n-1) = k,
    \]
so $\mathcal{L}_n \subseteq \mathcal{L}_{n+2}$.

For even $k$ with $n^2/2 < k \leq (n+2)^2/2$, we have $k-2(n+1)\in \mathcal{L}_n$.  Let $Y'_{k-2n-2}$ denote the vector obtained by bordering $y_{k-2n-2}$ with $-1$ on the top and by $1$ on the bottom, i.e., $Y'_{k-2n-2} = \left[ \begin{array}{c} -1 \\ y_{k-2n-2}
\\ -1 \end{array} \right]$. Then 
    \[
        x\trans R_{n+2}^{-1} Y'_{k-2n-2} = (n-1) + x\trans R_n^{-1}y_{k-2n-2} + (n-1) = k,
    \]
so $k \in \mathcal{L}_{n+2}$.

Thus, for each $k \in \mathcal{L}(n)$ there is some choice of vector $y \in \{\pm 1\}^n$ so that $\det(M) = (1+k)^2$.
 \end{proof}

\section{Expected values associated with Seidel adjacency matrices}
\label{sec: average}
Let $m$ be a positive integer and $n=2m$.
Let $G$ be a graph on vertices $1,2,\ldots, 2m$.
A {\it matching} of $G$ is a collection of vertex disjoint edges of $G$.
If each vertex of $G$ is incident to an edge of a given matching, then the matching is {\it perfect}.
The set of all perfect matchings of $K_n$ is denoted by $\mathcal{M_n}$.
Given a perfect matching $M=\{ i_1j_1, \ldots, i_mj_m\}$ we may assume that 
the vertices in each edge are ordered so that $i_k<j_k$; and the edges are ordered
so that $i_1<i_2< \cdots < i_m$.  This ordering gives a permutation $i_1,j_1, \ldots, i_m, j_m$
of $1,2, \ldots, n$ that we denote by $\sigma_M$. 
Let $S=[s_{ij}]$ be an $n\times n$ skew-symmetric matrix.
The matrix $S$ determines a weight on $M$, namely, $\mbox{wt}_S(M)=s_{i_1,j_1}s_{i_2,j_2} \cdots s_{i_m,j_m}$.
The {\it Pfaffian of S} is denoted by $\mbox{Pf}\; A$ and is defined by 
\[ 
\mbox{Pf}\; S = \sum_{M\in \mathcal{M}} \mbox{sgn}(\sigma_M) \mbox{wt}_S(M).\]
It is well-known, see \cite{Cayley,HJ}, that 
\[ 
\det S = (\mbox{Pf}\; S)^2.
\]
In the remainder of this section we consider the set $\mathcal{S}_n$ of 
Seidel matrices for tournaments of size $n$ endowed with the uniform distribution; that is, each matrix in $\mathcal{S}_n$ occurs with probability $\frac{1}{2^{n \choose 2}}$.  For $i<j$, we let $k_{ij}$ be a variable that takes on either the value $1$ or value $-1$ each with probability $1/2$.  Thus the skew-symmetric matrix $K=[k_{ij}]$
where $k_{ii}=0$ and $k_{ji}=-k_{ij}$ for $i<j$, represents a random 
matrix in $\mathcal{S}_n$.  We will consider various random variables on $\mathcal{S}_n$: the determinant and the square of the determinant.  Additionally, for each multiset $\alpha$ of 
$\{(i,j): i<j\}$ we will consider the random variable
\[ r_{\alpha}= \prod_{(i,j) \in \alpha} {k_{ij}}^{m_{ij}},\] 
where $m_{ij}$ is number of times $(i,j)$ occurs in the multiset $\alpha$.

Note that 
\begin{equation}
\label{eq:expected}
\mathbb{E}(r_{\alpha}) = 
\left\{  \begin{array}{cl}
0 & \mbox{ if $(i,j)$ occurs in $\alpha$ with odd multiplicity, and }
\\
1 & \mbox{otherwise.}
\end{array} \right.
\end{equation}

If $(M_1,M_2, \ldots, M_k) \in \mathcal{M}^k$, then we set 
$\alpha_{(M_1,M_2,\ldots, M_k)}$ to be the multigraph whose edges are those in the union of $M_1, M_2, \ldots, M_k$ (see \Cref{fig: multigraphs}).
Note that 
\[ 
\mathbb{E}(\mbox{wt}_K(M_1) \mbox{wt}_K(M_2) \cdots \mbox{wt}_K(M_k))=\mathbb{E}(r_{\alpha_{(M_1,M_2,\ldots, M_k)}}),
\] 
and that (\ref{eq:expected}) implies  $\mathbb{E}(r_{\alpha_{(M_1,M_2,\ldots, M_k)}})$
is $1$ if and only if 
each edge in the union of $M_1$, \ldots, $M_k$ occurs an even number of times, and is zero otherwise. 

\begin{figure}
    \centering
    \includegraphics[width = \textwidth]{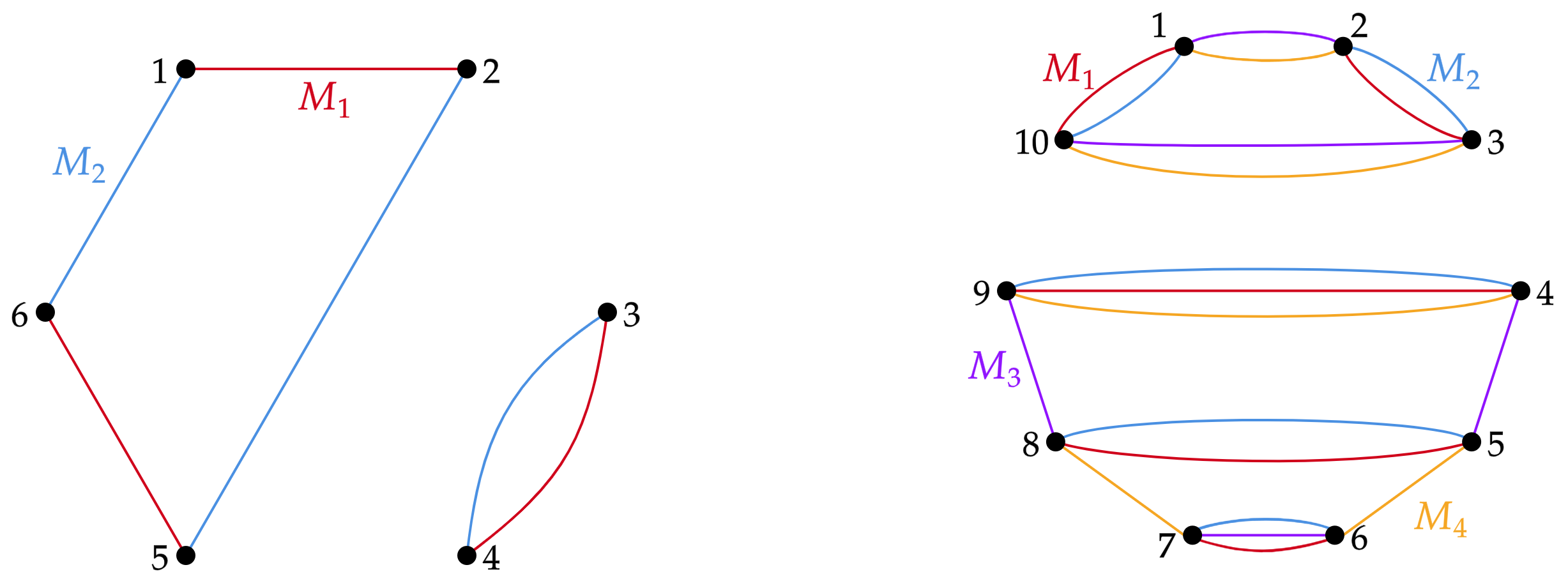}
    \caption{Examples of multigraphs formed as a union of perfect matchings.}
    \label{fig: multigraphs}
\end{figure}

We use this to calculate the expected value of $\det S$ over all Seidel matrices $S$ of tournaments of order $n$.  As is customary, $n!!$ is defined by $0!!=1!!=1$, 
and $n!!=n \cdot (n-2)!!$ for each integer $n\geq 2$.

\begin{theorem}\label{thm:expected}
Let $n=2m$ be an even positive integer. Then the expected value of $\det X$ over all Seidel matrices of tournaments of order $n$ is $\mathbb{E}(\det X) = (n-1)!!$.  
\end{theorem}

\begin{proof}
Observe that 
\begin{eqnarray}
\mathbb{E}(\det X: X\in \mathcal{S}_n) &= & 
 \frac{1}{2^{n \choose 2}} \sum_{ S \in \mathcal{S}_n} \det S
\\
& = & \frac{1}{2^{n \choose 2}} \sum_{S \in \mathcal{S}_n}  (\mbox{Pf}\,S)^2
\\
 &= &  \frac{1}{2^{n \choose 2}} \sum_{S\in \mathcal{S}_n}\sum_{M,N \in \mathcal{M}} \mbox{sgn}(\sigma_M)\mbox{sgn}(\sigma_N) \mbox{wt}_S(M) \mbox{wt}_{S}(N) \\
 & = & \sum_{M,N \in \mathcal{M}}  \mbox{sgn}(\sigma_M)\mbox{sgn}(
 \sigma_N) 
 \mathbb{E} (
  \mbox{wt}_X(M) \mbox{wt}_{X}(N)). \label{eq:sum}
 \end{eqnarray}
The union of the edges of the perfect matchings $M$ and $N$ form a multigraph $\alpha_{M, N}$  each of whose connected components is either a 
pair of  duplicate edges,  or form an even-cycle of length at least $4$.  By (\ref{eq:expected}),  $\mathbb{E}(\mbox{wt}_X(M) 
\mbox{wt}_X(N))$ is $1$ in the former case, and is $0$ in the latter case.  Hence, the sum in (\ref{eq:sum}) equals the number of perfect matchings of $K_n$. This is exactly $(n-1)!!$.
\end{proof}

Note this theorem implies that for each even $n$, $\mathcal{D}(n)$ 
contains a value that is at least $\sqrt{ (n-1)!!}$.  This bound grows super-exponentially 
in $n$. 

We next give a recursive formula for the expected value of $(\det S)^2$.

\begin{theorem}
\label{thm:square}
Let $n$ be an even integer. 
Let $z_n=\mathbb{E}( (\det X)^2)$ on $\mathcal{S}_n$.
Then $z_n=y_n \cdot (n-1)!!$, 
where $y_n$ is given by the recurrence
\begin{eqnarray*}
y_0&= & 1\\
y_2 & =& 1\\
y_n&=& (n-1)y_{n-2} + (2n-4)y_{n-4} \mbox{ for $n\geq 4$.}
\end{eqnarray*}
\end{theorem}

\begin{proof}
Let $\mathcal{L}_n$ denote the set of $4$-regular multigraphs on vertices $1,2,\ldots, n$ whose edges are labeled by $\{1,2,3,4\}$ 
such that each edge has even multiplicity, and edges of the same labels are not incident.  We claim that $\mathbb{E}((\det S)^2)$
is equal to the number of elements of $\mathcal{L}_n$.

Note that 
\begin{eqnarray}
\label{pfaff4}
\mathbb{E}((\det X)^2)&=& 
\mathbb{E}((\mbox{Pf}\,X)^4)\\
&=& \frac{1}{2^{n \choose 2}}
\sum_{S \in \mathcal{S}_n}\sum_{(M_1,M_2,M_3,M_4)\in \mathcal{M}^4} 
\prod_{i=1}^4 
\mbox{sgn}\;\sigma_{M_i} \cdot\mbox{wt}_S(M_i)
\\
& = & \sum_{(M_1,M_2,M_3,M_4) \in \mathcal{M}^4} \prod_{i=1}^4 \mbox{sgn}\; \sigma_{M_i} \mathbb{E}\left(\prod_{j=1}^4 \mbox{wt}_X(M_j)\right)
\end{eqnarray}
By (\ref{eq:expected}),  $\mathbb{E}(\prod_{j=1}^4 \mbox{wt}_X(M_j))$
is $1$ if the union of the $M_i$ forms a multigraph $\alpha_{(M_1,M_2, M_3, M_4)}$ in which each edge has even multiplicity and is $0$ otherwise. Thus, we may restrict our attention to the former case.

Note that each vertex of $\alpha_{(M_1, M_2, M_3,M_4)}$ has degree $4$.
Hence each vertex is incident to an edge of multiplicity 4, or 
two edges each of multiplicity two. It follows that the  connected components of $\alpha_{(M_1, M_2, M_3,M_4)}$  consist of even cycles of length at least 4 where each edge has multiplicity 2, or 
edges each of multiplicity $4$.  Note that this requires that
for each edge $e$ of multiplicity 4, $e$ is in each $M_i$; 
and for each edge $e$ of multiplicity $2$, $e$ is in two of the $M_i$ and these have the same set of edges in the cycle of 
$\alpha_{(M_1, M_2, M_3,M_4)}$ that contains $e$.  This implies that 
$\prod_{i=1}^4 \mbox{sgn}\; \sigma_{M_i}=1$.

Now label each edge of $M_i$ by $i$.  This gives a labelling of 
the edges of  the multigraph $\alpha_{(M_1, M_2, M_3,M_4)}$ by
elements in $\{1,2,3,4\}$ in such a way that no edges with the same label are incident to each other.  For an edge of multiplicity $4$,
we have edges of each label. For an edge $e$ of multiplicity $2$, there are two labels for this edge, and the other edges incident to  a given vertex of $e$ are labelled by the the other two values in $\{1,2,3,4\}$.  Thus, this labelling of 
$\alpha_{(M_1, M_2, M_3,M_4)}$ is in $\mathcal{L}_n$.  Conversely, given 
an element of $\mathcal{L}n$ the edges of color $i$ form a perfect matching $M_i$ of $K_n$ $(i=1,2,3,4)$ and the union of the edges of the $M_i$ give a multigraph each of whose edges has even multiplicity.  Hence $\mathbb{E}((\det S)^2)=|\mathcal{L}_n|$.

We next establish  recurrence for $z_n:=|\mathcal{L}_n|$.
There are three possibilities for elements $L$ of $\mathcal{L}_n$, 
and we give counts for each of the possibilities. Refer to \Cref{fig: expected_cases}.

\bigskip\noindent
{\bf Case 1.}  $n$ is incident to an edge in $L$ of multiplicity $4$.

\noindent
In this case, $L$ consists of a edge $n$--$j$ of multiplicity $4$
and an element $L'$ of $\mathcal{L}_{n-2}$ on the vertices $\{1, \ldots, n\} \setminus \{n,j\}$.  There are $(n-1)$ choices for $j$, and $z_{n-2}$ choices for $L'$.
Hence this case has a total of
\begin{equation}
\label{eq:count1}
(n-1)z_{n-2}
\end{equation}
possibilities for $L$.

\bigskip\noindent
{\bf Case 2.} $n$ is incident to a double $4$-cycle in $L$.

\noindent
In this case, $L$ consists of a double $4$-cycle containing $n$,
and an element of $\mathcal{L}_{n-4}$ on the vertices not in the double $4$-cycle.  There are $(n-1)(n-2)(n-3)/2$  ways to chose and order the 
vertices (other than $n$) for the double 4-cycle and $\binom{4}{2} = 6$ ways to label the double $4$-cycle.  Thus, the number of $L$ in Case 2 
equals 
\begin{equation}
\label{eq:count2}
3(n-1)(n-2)(n-3)z_{n-4}.
\end{equation}

\bigskip\noindent 
{\bf Case 3.} $n$ is on a double $\ell$-cycle in $L$ with $\ell \geq 6$. 

Let $i$ and $j$ be the neighbors of $n$ in $L$ with $i>j$
and let $k$ be the other neighbor of $j$ in $L$.   Deleting the vertices $n$ and $j$ (and the edges incident to them)
from $L$ and inserting a double edge between $i$ and $k$ with labels the same as the labels in $L$ on the edge $i$--$n$
gives a bijection between the $L$ satisfying Case 3, and the 
elements of $\mathcal{L}_{n-2}$ (on $\{1,2, \ldots, n\} \setminus
\{n,j\}$ for which $i$ is not on an edge of multiplicity $4$.

The number of elements in $\mathcal{L}_{n-2}$ having $i$ on an edge of multiplicity $4$ is $(n-3)z_{n-4}$.  Hence, there are 
$(z_{n-2} -(n-3)z_{n-4})$ labelled multigraphs $L$ that satisfy 
Case 3 for the given $i$ and $j$.  As there there $(n-1)(n-2)$
choices for $i$ and $j$, the total number $L$ from Case 3 is 
\begin{equation}
\label{eq:count3}
(n-1)(n-2)(z_{n-2}-(n-3)z_{n-4}).
\end{equation}

Putting together Cases 1-3, (\ref{eq:count1})--(\ref{eq:count3})
imply that 
\begin{eqnarray*} 
z_n
&= &(n-1)z_{n-2} +3(n-1)(n-2)(n-3)z_{n-4} + 
(n-1)(n-2)(z_{n-2} -(n-3)z_{n-4})\\
&= & (n-1)^2z_{n-2}  + 2(n-1)(n-2)(n-3) z_{n-4})
\end{eqnarray*}
Now set $y_n=\frac{z_n}{(n-1)!!}$.
Then the last equation simplifies to 
\[y_n= (n-1)y_{n-2} +(2n-4)y_{n-4}.\]
\end{proof}

\begin{figure}
    \centering
    \includegraphics[width = .8\textwidth]{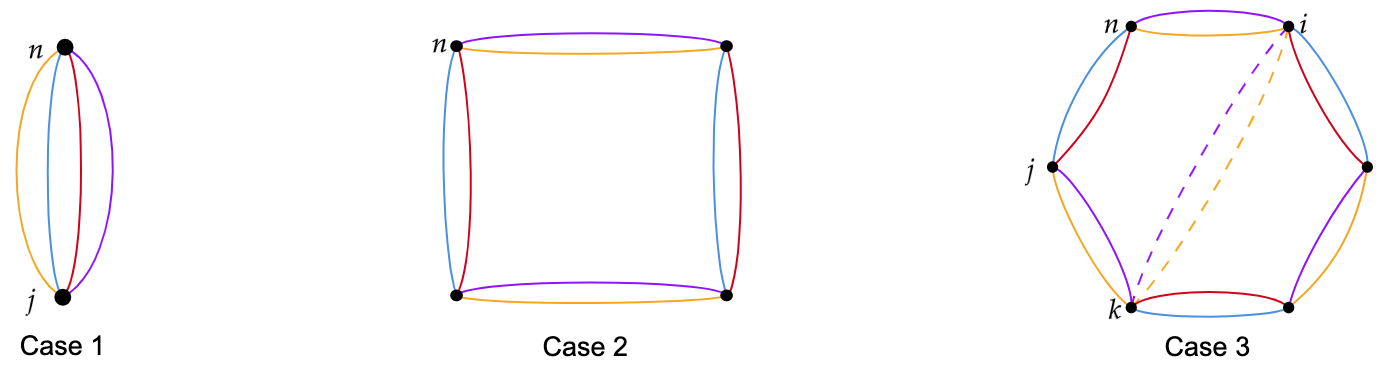}
    \caption{Types of components containing vertex $n$ as in the proof of \Cref{thm:square}.}
    \label{fig: expected_cases}
\end{figure}

Recall that the {\it variance} of a random variable $f$
on a probability distribution is defined by $\mathbb{V}(f)= 
\mathbb{E}(f^2)-(\mathbb{E}(f))^2$.

\begin{cor}
    Let $n$ be an even integer, and $y_n$ defined as in {\rm \cref{thm:square}}.
Then the variance of the random variable $f= \det S$
over the uniform distribution of $\mathcal{S}_n$ 
is given by 
\[ 
\mathbb{V}(f)= (n-1)!!(y_n - (n-1)!!). \]
\end{cor}

\begin{proof}
This follows from definition of variance,
\cref{thm:expected} and \cref{thm:square}.
\end{proof}

The table below gives various the first $7$ values of $y_n$ and $z_n$.  The values of $z_2$, $z_4$, $z_6$ and $z_8$ have been confirmed by direct computation of the average sum of the squares of the determinants of Seidel matrices. 
\begin{center}
\begin{tabular}{r|r|r}
$n$ & $y_n$ & $z_n$\\ \hline
2& 1& 1\\ 
4& 7& 21\\ 
6& 43& 645\\ 
8& 385& 40425\\ 
10& 4153& 3924585\\ 
12& 53383& 554916285\\ 
14& 793651& 107250027885
\end{tabular}
\end{center}

\section{Maximum value of $\mathcal{D}(n)$}
\label{sec: large}
Let $M(n) = \max\; \mathcal{D}(n)$. 
For $n$ even, \cref{thm: quadratic} 
demonstrates constructively that  $M(n)\geq (n^2/2+1)^2$. \cref{thm:expected} gives a non-constructive proof 
that $M(n)$ grows super-exponentially in $n$. 
 We can also bound $M(n)$ from above. Let $S$ be an $n\times n$ Seidel tournament matrix with $n$ even. Then each column of $S$ has Euclidean length $\sqrt{n-1}$, and Hadamard's inequality (see 2.1.P23 of \cite{HJ}) implies that $\det S\leq (n-1)^{n/2}$. Moreover, as equality occurs in Hadamard's inequality if and only if $S$ has mutually orthogonal rows, $\det  S = (n-1)^{n/2}$ if and only if $S$ is a skew-conference matrix. Hence we have the following. 

\begin{theorem}\label{th:max}
    For each integer $n$,  
        \[
            M(n) \leq (n-1)^{n/4}.
        \]
    Moreover, equality holds if and only if there is a skew-conference matrix of order $n$.
\end{theorem}

A tournament $T$ is \textit{doubly-regular} provided the number of vertices in $T$ dominated by two vertices $i$ and $j$ is independent of the choice of vertices $i\neq j$.  Necessarily, if $T$ is doubly-regular with more than 3 vertices, then $n\equiv 3 \bmod 4$. Given a doubly-regular tournament $T$ of order $n-1$, we can produce an $n\times n$  matrix by ``bordering" the Seidel matrix $S$ of $T$ as follows: take 
    \[
        S' = \left[\begin{array}{c|c}
                0 & j \\
                \hline
                -j\trans & S
            \end{array}\right],
    \]
where $j$ denotes the $1\times (n-1)$ vector of all ones. It is not difficult to show that $S'$ is a skew-conference matrix
(see (\cite{Reid-Brown}). 

When $n\equiv 3 \bmod 4$ is a power of a prime, the \textit{quadratic residue tournament}, $Q_n$, whose vertices are the elements of the field $\mbox{\rm GF}(n)$ and $ab$ is an arc if and only if $b-a$ is a square in the field, is an example of a doubly-regular tournament. Thus $M(n) = (n-1)^{n/4}$ infinitely often. It is conjectured that for all $n \equiv 3 \bmod 4$, there exists a doubly-regular tournament of order $n$ (see e.g., \cite{Reid-Brown}). 

We use existence of skew-conference matrices and the following number-theoretic result to show constructively that, eventually, the growth of $M(n)$ is super-exponential. 

\begin{theorem}[Special case of Theorem 3, \cite{BHP}] \label{primes} \
    For $\epsilon>0$ and $x$ sufficiently large with respect to $\epsilon$, the interval $[x-x^{0.55+\epsilon},x]$ contains primes congruent to $3$ modulo $4$.
\end{theorem}

In particular, there exists some $N\in \mathbb{N}$ such that whenever $n>N$ the interval $[n - n^{.6}, n]$ contains a prime $p \equiv 3\mod 4$. There is a quadratic residue tournament $Q_p$ of order $p$, and the corresponding $(p+1)\times (p+1)$ bordered Seidel matrix $S_p$ satisfies $\sqrt{\det(S_p)} = p^{(p+1)/4}\geq (n-n^{.6})^{(n-n^{.6}+1)/4}$. By \Cref{cor:contain} there is some order $n$ tournament achieving the same determinant as $Q_p$, so $M(n) \geq (n-n^{.6})^{(n-n^{.6}+1)/4}$. This is summarized in the following. 

\begin{theorem}\label{th: super-exponential}
    There exists an $N> 0$ such that $n\geq N$ implies that 
        \[
            M(n) \geq (n-n^{.6})^{(n-n^{.6}+1)/4}.
        \]
\end{theorem}

\Cref{thm:expected} implies that $M(n)\geq \sqrt{(n-1)!!}.$ One can verify via Stirling's approximation and standard analysis that the bound in \Cref{th: super-exponential} is stronger.

When $n \equiv 2 \bmod 4$ there is no skew-conference matrix of order $n$, so we cannot achieve the bound in \Cref{th:max}. However, upper bounds for largest determinants of such Seidel tournament matrices have received considerable attention.

\begin{theorem}\cite{AAFG, AF,GS}\label{thm: n=2mod4 skew determinants}
    For $n\equiv 2 \bmod 4$ we have
    \begin{itemize}
        \item[\rm (a)]  $M(n) \leq (2n-3)^{\frac{1}{2}}(n-3)^{\frac{n-2}{4}}$; 
        \item[\rm (b)] equality holds in (a) if and only if there exists a Seidel tournament matrix of order $n$ with 
            \[
                SS^{\trans}= S^{\trans}S= \left[ \begin{array}{c|c} L & O \\ \hline O & L \end{array} \right],
            \]
        where $L=(n-3)I+2J$; 
        \item[\rm (c)] equality holds in (a) if and only if there exists an $(n-1) \times (n-1)$ tournament matrix $A$ with characteristic polynomial 
        \[ (x^3-(2t-1)x^2 -t(4t-1))(x^2+x+t)^{2t-1}, \mbox{ where $t=(n-2)/4$;}  \] 
        \item[\rm (d)] if equality holds in (a), then $2n-3$ is a perfect square; 
        \item[\rm (e)] equality is known to hold for $n=6$, $14$, $26$, $42$. 
    \end{itemize}
\end{theorem}

We offer a lower bound on $M(n)$ in the case that there exists a skew-conference matrix of order $n+2$. Note that the bound in \Cref{thm: n=2mod4 skew determinants}(a) dominates our bound, and for $n\geq 4$ it is strictly larger than our bound. 
The fact that deleting any two rows and the same columns of an $(n+2) \times (n+2)$ skew-hadamard matrix
results in a matrix of determinant $(n+1)^{(n-2)/4}$ implies the following, and was shown by  Peter Cameron in \cite{CameronBlog}.
For completeness, we include a proof here. 
\begin{theorem}\label{thm: SCM-2}
    If there exists a skew-conference matrix of order $n+2$, then 
    \[ M(n) \geq (n+1)^{(n-2)/4}.\]
\end{theorem}

\begin{proof}
    Assume that a skew-conference matrix $S=[s_{k\ell}]$ of order $n+2$ exists.  Then $-S^2=S^TS=(n+1)I$, and the eigenvalues of $S$ are $\sqrt{n+1}i$ and $-\sqrt{n+1}i$, each of geometric multiplicity $(n+2)/2$.

    Let $\hat{S}$ be the matrix obtained by deleting the first two rows and columns of $S$. As deleting two rows and columns of a matrix decreases the nullity by at most $2$, both $\sqrt{n+1}i$ and $-\sqrt{n+1}i$ are eigenvalues of 
    $\hat{S}$ of geometric multiplicity at least $(n-2)/2$.
    This leaves two other eigenvalues, 
    $\lambda$ and $-\lambda$, since the eigenvalues of $\hat{S}$ are purely imaginary and occur in complex conjugate pairs. As $S$ is a normal matrix, 
    \[ 
    n(n-1)= \sum_{k=1}^n\sum_{\ell=1}^n |s_{k\ell}|^2=
    \frac{n-2}{2} (\sqrt{n+1})^2 + \frac{n-2}{2} (-\sqrt{n+1})^2 + |\lambda|^2 + |-\lambda|^2
    \]
It follows that $-\lambda^2=|\lambda|^2=1$. Furthermore, since the determinant of a matrix is the product of its eigenvalues,  $\det \hat{S}= (n+1)^{(n-2)/4}$.

\end{proof}

\section{Gaps} \label{sec: gaps}
We now derive an upper bound on the determinant of 
a Seidel matrix having a pair of non-orthogonal rows.
It makes use of the following well-known result \cite{Fischer}, (see also \cite[Statement 7.8.3]{HJ}).

\bigskip\noindent
{\bf Fischer's inequality.}
\\
{\it 
Let $A$ be an $n\times n$, real symmetric positive definite matrix of the 
form 
    \[ 
        A = \left[ \begin{array}{c|c} B & C \\ \hline
                                    C\trans & D 
                    \end{array}\right] 
    \]
where $B$ and $D$ are square. 
Then 
    \[ 
        \det A \leq \det B \det D.
    \]
}

\begin{theorem}\label{thm: bound on non-skew conf det}
Let $S$ be an $n\times n$
Seidel matrix of a tournament with $n$ even. Either $S$
is a skew-conference matrix or 
\[
\det S\leq (n-1)^{\frac{n-2}{2}} \sqrt{(n-1)^2-4} .
\]
\end{theorem}

\begin{proof}
Let $S$ be an $n\times n$ Seidel matrix of a tournament and let $A = SS\trans$. 
Assume that $S$ is not a skew-conference matrix.
Then the rows of  $S$ are not mutually orthogonal. 
Thus some off-diagonal entry of $A=[a_{ij}] $ is nonzero. 
Without loss of generality we may assume $a_{n-1, n} \neq 0$. 
By parity $|a_{n-1, n}| \geq 2$. 
Partition $A$ as 
 \[
 A= \left[ \begin{array}{c|c} \hat{A} & C \\ \hline
 C\trans& \begin{array}{cc} n-1 & a_{n-1,n} \\ a_{n-1,n} & n-1 
 \end{array} \end{array} \right].
 \]
 
 By Fischer's inequality, we have
 \begin{eqnarray*}
 \label{eq:Fischer}
    (\det S)^2 &=& \det A \\
                &\leq &\det \hat{A} \cdot \det 
                        \left[ \begin{array}{cc} n-1 & a_{n-1,n} \\ a_{n-1,n} & n-1 \end{array} \right] \\
                &=& \det \hat{A} \cdot ((n-1)^2 - a_{n-1,n}^2) \\
                & \leq & \det \hat{A} \cdot ((n-1)^2 - 4) \\
                & \leq & (n-1)^{(n-2)} ((n-1)^2-4).
\end{eqnarray*}
The last inequality comes from the following observation.
Note that $\hat{A}=Y\trans Y$ where $Y$ is the matrix formed by the 
first $n-2$ columns of $S\trans$. There exists an orthogonal matrix $Q$
and an $(n-2) \times (n-2)$ matrix $Z$ 
such $QY= \left[ \begin{array}{c} Z \\O \end{array} \right]$.
Thus $\hat{A}= Y\trans Q\trans QY = \left[ \begin{array}{cc} 
Z\trans & O \end{array} \right] \left[ \begin{array}{c} Z \\O \end{array} \right]= Z\trans Z$. Hence, $\det \hat{A}= (\det Z)^2$. 
Since $Q$ is orthogonal, each column of $Z$ has squared length 
equal to $n-1$, and Hadamard's inequality applied to $
Z$ yields $(\det Z)^2 \leq (n-1)^{n-2}$.
\end{proof}

\begin{cor}
Let $m$ a positive integer divisible by $4$
with $m <(n-1)^{(n-8)/4}.$
There is no $n\times n$ Seidel matrix $S$
with $\sqrt{\det S} = (n-1)^{n/4}-m$.
In particular, if $m \geq 12$ and there exists an $n\times n$
skew-conference matrix then $\mathcal{D}(n)$ contains a gap between its two largest elements. 
\end{cor}

\begin{proof}
Define $g_n= (n-1)^{n/4}- (n-1)^{(n-2)/4}\left( (n-1)^2 - 4\right)^{1/4}$.
Then 
    \begin{eqnarray*}
        g_n&=&(n-1)^{(n-2)/4}(\sqrt{n-1} - \sqrt{\sqrt{n^2-2n-3}})\\ 
            &= & (n-1)^{(n-2)/4} \frac{n-1 - \sqrt{n^2-2n-3}}{\sqrt{n-1} + \sqrt{\sqrt{n^2-2n-3}}}\\
            &=&  (n-1)^{(n-2)/4} \frac {(n-1)^2 - (n^2-2n-3)}{\left(\sqrt{n-1} - \sqrt{\sqrt{n^2-2n-3}}\right)  \left(n-1 +\sqrt{n^2-2n-3}\right)} \\ 
            & \geq &  (n-1)^{(n-2)/4} \frac{4}{ 2 \sqrt{n-1}
            \cdot (n-1)}\\ 
            & = &    (n-1)^{(n-8)/4}.
\end{eqnarray*}

By \Cref{thm: bound on non-skew conf det} if $S$ is a Seidel tournament matrix which is not skew-conference, then 
    \[
        \sqrt{\det S} < (n-1)^{(n-2)/4}\left( (n-1)^2 - 4\right)^{1/4} = (n-1)^{n/4} - g_n \leq (n-1)^{n/4} - (n-1)^{(n-8)/4}.
    \]
Hence, for each positive integer $m$ with $m \leq (n-1)^{(n-8)/4}$
there is no $n\times n$ Seidel matrix $S$ with 
$\sqrt{\det S}=(n-1)^{n/4}-m$. 

For $n\geq 12$ we have $(n-1)^{(n-8)/4}>2$, and thus the claimed gap in $\mathcal{D}(n)$ exists. 
\end{proof}

Our argument depends on the existence of a skew-conference matrix. We suspect that this is not a necessary assumption.

\begin{conj}
    There is a gap in $\mathcal{D}(n)$ for all even $n\geq 8$.
\end{conj}
For $n \in \{8, 10, 12\}$ the conjecture can be verified via computation in SageMath (see \Cref{fig: determinant sets}).

\section{Characteristic Polynomials of Seidel Matrices} \label{sec: charpoly}
In previous sections of this paper we investigated possible values of the determinant of an $n\times n$ Seidel tournament matrix $S$. In this section we broaden our scope to consider possible characteristic polynomials of such matrices, $c_S(x)$, with a particular focus on principal minors of skew-conference matrices. Note that $c_S(x) = \det(xI - S)$, so $\det S = c_S(0)$. In the case that $S$ has odd order, $\det S = 0$. All possible characteristic polynomials of $n \times n$ Seidel matrices, for $n \leq 6$, are listed in \Cref{fig: charpoly sets} as computed in SageMath. We denote the set of characteristic polynomials of Seidel matrices of order $n$ by $\mathcal{CP}(n)$.
    \[
        \mathcal{CP}(n) = \{ p(x) \mid p(x) \mbox{ is the characteristic polynomial of a Seidel matrix of order }n\}.
    \]

\begin{figure}[h]
    \centering
    \begin{tabular}{l|l}
     $n$&  $\mathcal{CP}(n)$\\
     \hline
     1&$x$\\
     \hline
     2& $x^2+1$\\
     \hline
     3& $x^3+3x$\\
     \hline
     4& $x^4 + 6x^2 + 1$\\
      & $x^4 + 6x^2 + 9$\\
      \hline
     5& $x^5 + 10x^3 + 5x$\\
      & $x^5 + 10x^3 + 21x$\\
      \hline
     6& $x^6 + 15x^4 + 15x^2 + 1$\\
           & $x^6 + 15x^4 + 47x^2 + 1$\\
      & $x^6 + 15x^4 + 39x^2 + 9$\\

      & $x^6 + 15x^4 + 55x^2 + 25$\\
      & $x^6 + 15x^4 + 63x^2 + 49$\\
      & $x^6 + 15x^4 + 63x^2 + 81$
\end{tabular}
    \caption{Characteristic polynomials of $n\times n$ Seidel matrices for $n\leq 6$.}
    \label{fig: charpoly sets}
\end{figure}

\cref{fig: charpoly7} lists the ordered pair of the  $x^3$ and $x^1$ coefficients of elements in $\mathcal{CP}(7)$. Thus, while there are $2^{21}$ Seidel matrices of order $7$, there are only 11 possible characteristic polynomials. \cref{fig: charpoly8} lists the $50$ triples of the $x^4$, $x^2$, and $x^0$ coefficients of  $2^{28}$ matrices in $\mathcal{CP}(8)$. 

\begin{figure} 
\[\begin{array}{rrrrrr}
(35, 7) & 
(99, 7)&
(83, 23) & 
(67, 39) &
(115, 55)& 
(99, 71) \\
(115, 119) &
(99, 135)&
(115, 183)&
(131, 231) & 
(147, 343) 
\end{array}
\]
\caption{Coefficients of $x^3$ and $x^1$ of Seidel matrices of order $7$.}
    \label{fig: charpoly7}
\end{figure}

\begin{figure}
\[
\begin{array}{rrrrr}
(70, 28, 1) & 
 (198, 28, 1)&
(166, 60, 1) &
(134, 92, 1)&
 (198, 156, 1)\\ 
(166, 188, 1)&
(198, 284, 1) &
(198, 412, 1)& 
(230, 508, 1) &
(142, 76, 9)\\
(110, 108, 9)& 
(174, 172, 9) &
(206, 268, 9) &
(174, 300, 9) &
 (206, 396, 9)\\ 
 (238, 620, 9) &
(190, 140, 25)&
(158, 172, 25)&
(190, 268, 25)&
 (222, 364, 25)\\
(222, 492, 25) &
(182, 252, 49) &
(214, 348, 49) &
(214, 476, 49) &
 (246, 700, 49)\\
(182, 348, 81) &
(150, 252, 81)&
(214, 316, 81) & 
 (214, 444, 81) &
 (246, 668, 81)\\
 (190, 300, 121)&
 (222, 396, 121)&
(222, 524, 121) &
 (254, 748, 121) &
 (206, 364, 169) \\
 (238, 588, 169)&
 (198, 444, 225) &
 (230, 540, 225)&
 (230, 476, 289)&
 (230, 604, 289)\\ 
(238, 652, 361)&
 (222, 588, 441)&
 (254, 812, 441)&
  (246, 732, 529)&
   (246, 764, 625)\\
 (222, 684, 729)&
 (262, 924, 961) &
 (262, 924, 1089)&
 (270, 1036, 1225)&
 (294, 1372, 2401) 
\end{array} 
\]
\caption{Coefficients of $x^4$, $x^2$, $x^0$ of Seidel matrices of order $8$.}
    \label{fig: charpoly8}
\end{figure}

We now give some  basic properties of characteristic polynomials of Seidel matrices of tournaments.
\begin{proposition}
\label{prop:simple}
Let $S$ be the Seidel matrix of a tournament of order $n$.
Then the characteristic polynomial $c_S(x)$ satisfies the following.
\begin{itemize}
\item[\rm (a)] The  nonzero roots of $c_S(x)$ are purely imaginary and occur in complex conjugate pairs;
\item[\rm (b)] $0$ is a root of $c_S(x)$ if and only if $n$ is odd;
\item[\rm (c)] each coefficient of $x^{n-k}$ is $0$ if $n$ and $k$ have different parity, and is at least $n \choose k$ if $n$ and $k$ have the same parity; and
\item[\rm (d)] if $\hat{S}$ is switching equivalent to $S$, 
then $c_S(x)= c_{\hat{S}}(x). $
\end{itemize} 
\end{proposition}

\begin{proof}
Statement (a) comes from the fact that the eigenvalues of a real skew-symmetric matrix are purely imaginary and occur in complex conjugate pairs. 

Statements (b)--(d) follow from the facts that each even order Seidel matrix has determinant at least 1, each odd order skew symmetric matrix has determinant 0, and for $k \in \{1, \ldots, n\}$ 
the coefficient of $x^{n-k} $ in $c_M(x)$  equals $(-1)^ks_k$,
where $s_k$ is the sum of the determinants of all $k\times k$
principal minors of $M$.
\end{proof} 

We note that equality holds in (c) of \cref{prop:simple} when $k=n-2$. We next turn our attention to skew-conference matrices. 
We make use of Jacobi's determinantal identity (see \cite[Statement 0.8.4]{HJ}). This identity relates the determinant of a principal submatrix 
of an invertible matrix $A$ to the complementary submatrix of $A^{-1}$.

\bigskip\noindent
{\bf Jacobi's determinantal identity.}
\\
Let $A$ be an invertible matrix. Let $A[\alpha]$ be the principal submatrix of $A$ whose rows have index in $\alpha$, and let 
$A^{-1}(\alpha)$ be the principal submatrix of $A^{-1}$ whose rows have index not in $\alpha$. 
Then 
\[\det A \cdot (\det A^{-1}(\alpha))=  \det A[\alpha].\]

\begin{cor}
\label{cor:Jacobi}
Let $S$ be an $n\times n$ skew-conference matrix, 
and let $\alpha$ be a subset of $\{1,2,\ldots, n\}$ 
of cardinality $k \leq n/2$. 
Then 
\[ (x^2+n-1)^{n/2-k}   c_{S[\alpha]}(x)=c_{S(\alpha)}(x). \]
\end{cor}

 \begin{proof}
Let $A=xI-S$. Note 
\[(xI-S)(xI+S)= x^2I+SS^T = (x^2+n-1) I.\] 
Thus, as a matrix in $\mathbb{R}(x)$, 
$A$ is invertible, and $A^{-1}= \frac{1}{x^2+n-1}(xI+S)$.

By Jacobi's determinantal identity, 
$\det (xI-S) \cdot \det A^{-1} (\alpha)= \det A[\alpha]$
So 
\[ 
(x^2+n-1)^{n/2} \frac{1}{(x^2+n-1)^{n-k}} \det (xI+S(\alpha)) = \det (xI-S[\alpha]). \]
Thus
\[ c_{S(\alpha)}(x) = c_{S(\alpha)\trans}(x) = c_{-S(\alpha)}(x) = c_{S[\alpha]} (x)\cdot (x^2+n-1)^{n/2-k}.
 \]
\end{proof}

We note that \Cref{cor:Jacobi} implies that if $S$ is an $n\times n$ skew-conference matrix and $\hat{S}$ is an $\ell \times \ell$ principal submatrix of $S$ with $\ell \geq \frac{n}{2}$, then $\frac{c_{\hat{S}}(x)}{{(x^2+n-1)}^{\ell}}$ is a polynomial in $\mathcal{CP}(n-\ell)$.

\begin{cor}
\label{cor:skewresults}
Let $S$ be an $n\times n$ skew-conference matrix. 
Then 
\begin{itemize}
\item[(a)] Each principal $\ell \times \ell$ submatrix of $S$
with $\ell > n/2$ has $\sqrt{n-1}i$ as an eigenvalue of multiplicity 
at least $\ell -n/2$.
\item[(b)] If $\hat S$ is a $k\times k$ principal submatrix of $S$
that is also a skew-conference matrix, then $n \geq 2k$.
\item[(c)] If $\alpha$ is a subset of $\{1,\ldots, n\}$ of 
cardinality $n/2$, then $S[\alpha]$ and its complement $S(\alpha)$
have the same characteristic polynomials.
\end{itemize}
\end{cor}

\begin{cor}
If there exists an $n\times n$ skew-conference matrix
having a $k\times k$ principal submatrix of determinant $d^2$,
and $k\neq n/2$,
then $d \cdot (n-1)^{n/4-k/2} \in \mathcal{D}(n-k)$. 
\end{cor}

Furthermore, for any Seidel tournament matrix the eigenvalues of each principal submatrix of $S$ \emph{interlace} with the eigenvalues of $S$. 

\bigskip
\noindent\textbf{Interlacing.}\\
    Let $A$ be an $n\times n$ real skew-symmetric matrix with eigenvalues $\{\lambda_j i\}_{j=1}^n$ so that $\lambda_1 \geq \dots \geq \lambda_n$.  Let $B$ be an $m\times m$ principal submatrix of $A$ with eigenvalues $\{\theta_j i\}_{j=1}^m$ so that $\theta_1\geq \dots \geq \theta_m$.
   Then 
        \[
            \lambda_{n-m+j} \leq \theta_j \leq \lambda_{j}.
        \]

This follows from an analogous result for real symmetric matrices (see \cite[Theorem 9.1.1]{GodsilAC}) and the fact that for any skew-symmetric matrix $S$, the eigenvalues of $S$ are roots of the eigenvalues of the symmetric matrix $S^2$. As a result, if $S$ is an $n\times n$ skew-conference matrix the eigenvalues of the principal submatrices of $S$ have modulus bounded by $\sqrt{n-1}$. This gives an obstruction to a matrix embedding as a principal submatrix of a skew-conference matrix.

\begin{proposition}
    Let $S$ be an $n\times n$ skew-conference matrix and let $R_k$ be the Seidel matrix of a transitive tournament on $k$ teams. For every constant $c>0$ there exists $N_c$ so that whenever $n>N_c$, $R_{n/c}$ is not a principal submatrix of $S$.
\end{proposition}

\begin{proof}
    By \cite{Ito} the eigenvalues of $R_k$ are $\frac{1+\zeta_k}{1-\zeta_k}$ where $\zeta_k$ is a $k$-th root of unity. Thus there exists an eigenvalue $\lambda_k$ of $R_k$ with modulus $|\lambda_k| = \frac{\sin(\pi/k)}{1-\cos(\pi/k)}$. As $k\to \infty$ this approaches $\frac{(\pi/k)}{(\pi/k)^2} = \frac{k}{\pi}$. Hence for $k\geq \frac{n}{c}$ there exists $N_c$ so that $n>N_c$ implies that $|\lambda_k|>\sqrt{n-1}$. By the interlacing, $R_k$ is not a  principal submatrix of an $n\times n$ skew-conference matrix where $n \geq N_c$.
\end{proof}

In particular, we have $N_2 \approx 9.5, N_3 \approx 21.9, N_4 \approx 39.$

Suppose there exists some $k$ so that there is no Seidel matrix of order $2k+1$ with eigenvalues $\pm\sqrt{n-1}i$. By \Cref{cor:skewresults} then there is no skew-conference matrix of order $4k$. Hence one could disprove the Hadamard Conjecture by showing that there exists some odd $2k+1$ so that no Seidel matrix of order $2k+1$ has eigenvalues $\pm i\sqrt{4k-1}$. However, this is never the case.

\begin{theorem}\label{thm: HC1}
    For each positive integer $k$ there exists a $(2k+1) \times (2k+1)$ Seidel matrix with eigenvalues $\pm\sqrt{4k-1}i$.
\end{theorem}

\begin{proof}
We construct $S$ explicitly. Let $R_{2k-1}$ denote the Seidel matrix of the transitive tournament on $2k-1$ teams. Take 
    \[
        S = \left[ \begin{array}{rr|rrrrr}
                    0 & 1 & -1 &1 &-1 &\cdots &-1 \\ 
                    -1& 0 & 1 &-1 & 1 & \cdots &1 \\ \hline
                    1 &-1 &&&&&\\
                    -1 &1 &&&&&\\
                    1 &-1 &&& R_{2k-1} &&\\
                    \vdots &\vdots &&&&&\\
                    1 &-1 &&&&&
            \end{array} 
            \right].
    \]
One can verify (e.g., by induction on $k$) that $S^2$ is a symmetric matrix with first two rows as follows:
        \[
        S^2 = \left[ \begin{array}{rrrrrrr}
                    -2k & 2k-1 & 1 &-1 &1 &\cdots &1 \\ 
                    2k-1& -2k & 1 &-1 & 1 & \cdots &1 \\ 
                    * &* &*&*&*&*&*\\
                    \vdots &\vdots &\vdots&\vdots&\vdots&\vdots&\vdots\\
                    * &* &*&*&*&*&*
            \end{array} 
            \right].
    \]
Also,
\begin{eqnarray*}
\left[  \begin{array}{ccccc} 1 & -1 & 0 & \cdots & 0 \end{array} \right] S^2&=&
\left[ \begin{array}{ccccc} -4k+1 & 4k-1 & 0 & \cdots & 0 \end{array} \right]\\
&=& -(4k-1) \left[  \begin{array}{ccccc} 1 & -1 & 0 & \cdots & 0 \end{array} \right].
\end{eqnarray*}
Hence $-(4k-1)$ is an eigenvalue of $S^2$, and $\sqrt{4k-1}i$ is an eigenvalue of $S$. 
\end{proof}

Finally, we consider the expected value of the characteristic polynomial, which turns out to be related to the matching polynomial for the corresponding undirected graph. 

\begin{theorem}\label{thm:expectedcharpoly}
    For any integer $n$, the expected value of the coefficient $c_{2k}$ of $x^{n-2k}$ in the characteristic polynomial over all Seidel tournament matrices of order $n$ is the number of matchings of size $k$ (i.e., using exactly $k$ edges) in the complete graph $K_n$. 
\end{theorem}

\begin{proof}
Note that the number of matchings with exactly $k$ edges on the complete graph with $n$ vertices is the sum of the number of perfect matchings over all vertex sets $\alpha$ having $2k$ vertices.  So, applying \Cref{thm:expected} we have
\[\mathbb{E}( c_{2k}) = \sum \mathbb{E}( \det K[\alpha])=\sum ((2k-1)!!)=\binom{n}{2k} (2k-1)!!,\]
where the sums are taken over all subsets of $\{1,2, \dots, n\}$ of size $2k.$ The result now follows, since $\binom{n}{2k} (2k-1)!!$ is the number of matchings of size $k$ of the complete graph $K_n$.
\end{proof}

It follows that the expected characteristic polynomial over all Seidel torunament matrices of order $n$ is given by 
\[
c(x) = \sum_{k=0}^{\lfloor n/2 \rfloor} \binom{n}{2k} (2k-1)!! x^{n-2k}.
\]
This is related to the matching polynomial for $K_n$, also known as the probabilist's Hermite polynomial; indeed, the coefficients of the matching polynomial have the same absolute value but alternate in sign.

The preceding results on the expected value of characteristic polynomials over all possible Seidel matrices are not restricted to tournaments. One can extend the notion of Seidel matrices from tournaments (i.e., orientations of the complete graph) to orientations on arbitrary graphs as follows: For a graph $G$ on $n$ vertices, a {\it Seidel matrix of $G$}  is an $n\times n$ skew-symmetric matrix $S=[s_{ij}]$ such that  $s_{ij} \in \{ \pm 1\}$  if $ij$ is an edge of $G$, and $s_{ij}=0$ otherwise. Thus, $S$ represents the orientation of $G$ where the edge $ij$ of $G$
is oriented from $i$ to $j$ provided $s_{ij}>1$.

 More generally, the proofs of \Cref{thm:expected} and  \Cref{thm:expectedcharpoly} can be applied to any graph $G$ to obtain the following.

\begin{theorem}\label{thm: generalexpectedcharpoly}
    Let $G$ be any graph with $n$ vertices. Then the expected characteristic polynomial 
    over the uniform distribution of the Seidel matrices of the graph $G$ is
    \[c(x) = \sum_{k=0}^{\lfloor n/2 \rfloor} m_k x^{n-2k},\]
    where $m_k$ is the number of matchings of $G$ having exactly $k$ edges. 
\end{theorem}

We also remark that \Cref{thm: generalexpectedcharpoly} could be proved using \cite[Theorem 2.3]{HL}, which gives a combinatorial formula for the coefficients of the characteristic polynomial for a Seidel matrix. 

\bibliography{refs.bib}

\begin{thebibliography}{10}

\bibitem{AAFG}
V\'{\i}ctor \'{A}lvarez, Jos\'{e}~Andr\'{e}s Armario, Mar\'{\i}a~Dolores Frau, and F\'{e}lix Gudiel.
\newblock Determinants of {$(-1,1)$}-matrices of the skew-symmetric type: a cocyclic approach.
\newblock {\em Open Math.}, 13(1):16--25, 2015.

\bibitem{AF}
Jos\'{e}~Andr\'{e}s Armario and Mar\'{\i}a~Dolores Frau.
\newblock On skew {E}-{W} matrices.
\newblock {\em J. Combin. Des.}, 24(10):461--472, 2016.

\bibitem{BHP}
R.~C. Baker, G.~Harman, and J.~Pintz.
\newblock The exceptional set for {G}oldbach's problem in short intervals.
\newblock In {\em Sieve methods, exponential sums, and their applications in number theory ({C}ardiff, 1995)}, volume 237 of {\em London Math. Soc. Lecture Note Ser.}, pages 1--54. Cambridge Univ. Press, Cambridge, 1997.

\bibitem{BRoy}
Sudipto Banerjee and Anindya Roy.
\newblock {\em Linear algebra and matrix analysis for statistics}.
\newblock Chapman \& Hall/CRC Texts in Statistical Science Series. CRC Press, Boca Raton, FL, 2014.

\bibitem{BBCL}
Wiam Belkouche, Abderrahim Boussa\"{\i}ri, Abdelhak Cha\"{\i}cha\^{a}, and Soufiane Lakhlifi.
\newblock On unimodular tournaments.
\newblock {\em Linear Algebra Appl.}, 632:50--60, 2022.

\bibitem{BC}
Abderrahim Boussa\"{\i}ri and Brahim Chergui.
\newblock Skew-symmetric matrices and their principal minors.
\newblock {\em Linear Algebra Appl.}, 485:47--57, 2015.

\bibitem{BSTZ}
Abderrahim Boussa\"{\i}ri, Imane Souktani, Imane Talbaoui, and Mohamed Zouagui.
\newblock {$k$}-spectrally monomorphic tournaments.
\newblock {\em Discrete Math.}, 345(5):Paper No. 112804, 9, 2022.

\bibitem{CameronBlog}
Peter Cameron.
\newblock Conference matrices, 2011, \url{https://webspace.maths.qmul.ac.uk/p.j.cameron/csgnotes/conftalk.pdf}, 05-16-2024.

\bibitem{Cayley}
Arthur Cayley.
\newblock Sur les d\'{e}terminants gauches.
\newblock {\em Journal f\"{u}r die reine und angewandte Mathematik}, 38:93–96, 1849.

\bibitem{deCaen}
D.~de~Caen, D.~A. Gregory, S.~J. Kirkland, N.~J. Pullman, and J.~S. Maybee.
\newblock Algebraic multiplicity of the eigenvalues of a tournament matrix.
\newblock {\em Linear Algebra Appl.}, 169:179--193, 1992.

\bibitem{Fischer}
Ernst Fischer.
\newblock Über den {H}adamardschen determinentsatz.
\newblock {\em {A}rch. {M}ath. {U}. {P}hys.}, 13:32--40, 1907.

\bibitem{GodsilAC}
C.~D. Godsil.
\newblock {\em Algebraic combinatorics}.
\newblock Chapman and Hall Mathematics Series. Chapman \& Hall, New York, 1993.

\bibitem{GS}
Gary Greaves and Sho Suda.
\newblock Symmetric and skew-symmetric {$\{0,\pm 1\}$}-matrices with large determinants.
\newblock {\em J. Combin. Des.}, 25(11):507--522, 2017.

\bibitem{HJ}
Roger~A. Horn and Charles~R. Johnson.
\newblock {\em Matrix analysis}.
\newblock Cambridge University Press, Cambridge, second edition, 2013.

\bibitem{HL}
Y.~Hou and T~Lei.
\newblock Characteristic polynomials of skew-adjacency matrices of oriented graphs.
\newblock {\em Electron. J. Comb.}, 18:P156, 1--12, 2011.

\bibitem{Ito}
Keiji Ito.
\newblock The skew energy of tournaments.
\newblock {\em Linear Algebra Appl.}, 518:144--158, 2017.

\bibitem{KS}
Stephen~J. Kirkland and Bryan~L. Shader.
\newblock Tournament matrices with extremal spectral properties.
\newblock {\em Linear Algebra Appl.}, 196:1--17, 1994.

\bibitem{Moon}
John~W. Moon.
\newblock {\em Topics on tournaments}.
\newblock Holt, Rinehart and Winston, New York-Montreal, Quebec-London, 1968.

\bibitem{Reid-Brown}
K.~B. Reid and Ezra Brown.
\newblock Doubly regular tournaments are equivalent to skew {H}adamard matrices.
\newblock {\em J. Combinatorial Theory Ser. A}, 12:332--338, 1972.

\end{thebibliography}
\bibliographystyle{plain}
\end{document}